\documentclass{amsart}
\usepackage{mathrsfs}
\usepackage{amsfonts}
\usepackage{txfonts}
\usepackage{amssymb}
\usepackage{graphicx}
\usepackage[arrow,matrix]{xy}
\usepackage{amsmath,amssymb,amscd,bbm,amsthm,mathrsfs,dsfont}
\usepackage{amsmath,amscd}\input amssym.def
\usepackage{amsfonts,amssymb}
\usepackage{amsmath,amscd}\input amssym.def

\newtheorem{theorem}{Theorem}[section]
\newtheorem{lemma}[theorem]{Lemma}

\theoremstyle{definition}
\newtheorem{definition}[theorem]{Definition}

\newtheorem{corollary}[theorem]{Corollary}
\theoremstyle{remark}

\DeclareMathOperator{\Ext}{Ext} 
 \numberwithin{equation}{section}

\begin{document}

\title{On an example of $\delta$-Koszul algebras}

\author{L\"{U} Jia-Feng}
\address{Department of Mathematics, Zhejiang Normal University, Jinhua, Zhejiang, P.R. China 321004}
\email{jiafenglv@zjnu.edu.cn}



\date{September 25, 2009 and, in revised form, November 18, 2009.}

\begin{abstract}
The main purpose of this paper is to study a concrete example of
$\delta$-Koszul algebras, which is related to three questions raised
by Green and Marcos in \cite{GM}.
\end{abstract}
\maketitle
\section{Introduction}
It is well known that whether the Yoneda algebra of a graded algebra
is finitely generated or not is too complicated to be answered. As
an attempt to discuss this thesis, Green and Marcos introduced the
notion of {\it $\delta$-Koszul algebra} in \cite{GM} in 2005. In
particular, they finished the paper with three questions:

\begin{itemize}
\item For which functions $\delta: \mathbb{N}\rightarrow \mathbb{N}$ is
there a $\delta$-resolution determined algebra?
\item For which functions $\delta: \mathbb{N}\rightarrow \mathbb{N}$ is
there a $\delta$-Koszul algebra?
\item  Is there a bound $N_0\in \mathbb{N}$, such that if $A=A_0
\oplus A_1\oplus A_2\oplus\cdots$ is a $\delta$-Koszul algebra, then
the Yoneda algebra $E(A)=\bigoplus_{n\geq 0}\Ext_A^n(A_0,A_0)$ is
generated by $\Ext_A^0(A_0,A_0)$, $\Ext_A^1(A_0,A_0)$, $\cdots$,
$\Ext_A^{N_0}(A_0,A_0)$?
\end{itemize}

In this paper, we give a sufficient condition for the resolution map
$\delta$ such that there do exist $\delta$-resolution determined
algebras and $\delta$-Koszul algebras. Further, we give an explicit
procedures to construct concrete examples of $\delta$-resolution
determined algebras and $\delta$-Koszul algebras satisfying this
condition. It should be noted that such examples are a special class
of {\it almost Koszul algebras} introduced by Brenner, Butler and
King with the aim to find periodic resolutions for the trivial
extension algebras of path algebras of Dynkin quivers in bipartite
orientation (see \cite{ARS} and \cite{BBK} for the further details)
and give an answer to the third question introduced above.

Now let us introduce some notations and recall some definitions.

Throughout the whole paper, $\Bbbk$ denotes an fixed field,
$\mathbb{N}$ denotes the set of natural numbers. All the positively
graded $\Bbbk$-algebra $A=\bigoplus_{i\geq 0}A_i$ are assumed with
the following conditions:

\begin{itemize}
\item $A_0=\Bbbk\times\cdots\times\Bbbk$, a finite product of $\Bbbk$;
\item $A_i \cdot A_j = A_{i+j}$ for all $0\leq i, j<\infty$;
\item dim$_{\Bbbk}A_i<\infty$ for all $i\geq 0$.
\end{itemize}
\begin{definition}(\cite{GM})\label{1}
Let $A$ be a positively graded algebra. $A$ is called {\it
$\delta$-Koszul} provided the following two conditions:
\begin{enumerate}
\item The trivial $A$-module $A_0$ admits a minimal graded projective
resolution
$$ \quad  \cdots \rightarrow P_n\rightarrow \cdots \rightarrow
P_1\rightarrow P_0\rightarrow A_0\rightarrow 0,$$ such that each
$P_n$ is generated in a single degree, say $\delta(n)$ for all
$n\geq0$, where $\delta$ is a strictly increasing set function;

\item The Yoneda-Ext algebra, $E(A)=\bigoplus_{n\geq
0}\Ext_A^n(A_0,A_0)$, is finitely generated as a graded algebra.
\end{enumerate}

If $A$ only satisfies condition (1), we call $A$ a {\it
$\delta$-resolution determined algebra. }
\end{definition}
\medskip
\section{Main results}
We begin with
\begin{definition}\label{d1}
A set map $f: \mathbb{N}\rightarrow \mathbb{N}$ is called {\it good}
if and only if there exists $N_0\in\mathbb{N}-\{0\}$, such that
\begin{enumerate}
\item $f(i)=i$ for all $0\leq i< N_0$;
\item  $f(i)=f(i-N_0)+f(N_0)$ for all $i\geq N_0$. In particular, if
$N_0\geq 3$, then $f(N_0)=N_0+1$.
\end{enumerate}
\end{definition}

\begin{lemma}\label{lem1}
Let $\delta: \mathbb{N}\rightarrow \mathbb{N}$ be a good set map.
Then $A$ is a $\delta$-resolution determined algebra if and only if
$A$ is a $\delta$-Koszul algebra.
\end{lemma}
\begin{proof}
It is immediate from (Theorem 3.6, \cite{GM}) and Definition
\ref{1}.
\end{proof}

\begin{lemma}\label{lem3} Let $\delta: \mathbb{N}\rightarrow \mathbb{N}$ be a good set map. Then there exist $\delta$-resolution
determined algebras.
\end{lemma}
\begin{proof}
By hypothesis, $\delta$ satisfies
$\delta(i)=\delta(i-N_0)+\delta(N_0)$ for all $i\geq N_0$ and
$\delta(i)=i$ for $i=0,\;1,\;\cdots,\;N_0-1$, where
$N_0\in\mathbb{N}-\{0\}$. We divided the proof into three cases.

(i) If $N_0=1$, Koszul algebras are the desired $\delta$-resolution
determined algebras with $\delta(i)=i$ for all $i\geq 0$ and there
are a lot of Koszul algebras.

(ii) If $N_0=2$, $d$-Koszul algebras are the desired
$\delta$-resolution determined algebras, where the set function
$\delta$ is defined as
$$\delta(i)=\left\{\begin{array}{ll}
\frac{id}{2},  &  \mbox{$i \equiv 0 (\textrm{mod} 2)$,}\\
\frac{(i-1)d}{2}+1,  &  \mbox{$i \equiv 1 (\textrm{mod} 2)$.}
\end{array}
\right..
$$

(iii) If $N_0\geq 3$, let $\Gamma$ be the quiver:
$$\bullet^1\leftrightarrows^{\alpha_1}_{\beta_1}\bullet^2\leftrightarrows^{\alpha_2}_{\beta_2
}\bullet^3\leftrightarrows^{\alpha_3}_{\beta_3}\cdots\leftrightarrows\bullet^{N_0-1}\leftrightarrows^
{\alpha_{N_0-1}}_{\beta_{N_0-1}}\bullet^{{N_0}}.$$ Now let
$$A=\frac{\Bbbk\Gamma}{\langle\alpha_i\beta_i-\beta_{i+1}\alpha_{i+1},
\;\alpha_{i+1}\alpha_i,\;\beta_i\beta_{i+1}:
i=1,\;2,\cdots,{N_0}-2\rangle}.$$

Now we will compute out the minimal graded projective resolution of
the trivial $A$-module $A_0$ as follows.

Let $P_i$ denote the simple $A$-module related to the vertex $i$.

If $N_0=3$, then $\Bbbk^{\otimes 3}$ has the following minimal
graded projective resolution

$\cdots\rightarrow(A\oplus P_2)[6] \rightarrow(A\oplus P_2)[5]
\rightarrow A[4]\rightarrow (A\oplus P_2)[2] \rightarrow (A\oplus
P_2)[1]\rightarrow A\rightarrow\Bbbk^{\otimes 3}\rightarrow 0$.

If $N_0=4$, then $\Bbbk^{\otimes 4}$ has the following minimal
graded projective resolution

$\cdots\rightarrow(A\oplus P_2\oplus P_3)[7] \rightarrow(A\oplus
P_2\oplus P_3)[6] \rightarrow A[5]\rightarrow (A\oplus P_2\oplus
P_3)[3]\rightarrow (A\oplus P_2\oplus P_3)[2] \rightarrow (A\oplus
P_2\oplus P_3)[1]\rightarrow A\rightarrow\Bbbk^{\otimes
4}\rightarrow 0$.

By an induction, we get that the minimal graded projective
resolution of the trivial $A$-module $\Bbbk^{\otimes N_0}$ has the
following general form:

$\cdots\rightarrow (A\oplus (P_2\oplus\cdots\oplus
P_{N_0-1}))[N_0+2]\rightarrow A[N_0+1]\rightarrow
A[N_0-1]\rightarrow \cdots\rightarrow (A\oplus(P_2\oplus\cdots\oplus
P_{N_0-1})\oplus(P_3\oplus\cdots\oplus P_{N_0-2})\oplus\cdots\oplus
P_{\frac{N_0-1}{2}}\oplus P_{\frac{N_0+1}{2}}\oplus
P_{\frac{N_0+3}{2}})[\frac{N_0+1}{2}]\rightarrow
(A\oplus(P_2\oplus\cdots\oplus
P_{N_0-1})\oplus(P_3\oplus\cdots\oplus P_{N_0-2})\oplus\cdots\oplus
P_{\frac{N_0+1}{2}})[\frac{N_0-1}{2}]\rightarrow \cdots\rightarrow
(A\oplus(P_2\oplus\cdots\oplus
P_{N_0-1})\oplus(P_3\oplus\cdots\oplus P_{N_0-2}))[2] \rightarrow
(A\oplus(P_2\oplus\cdots\oplus P_{N_0-1}))[1]\rightarrow
A\rightarrow \Bbbk^{\otimes N_0}\rightarrow 0$ for $N_0$ being odd;

$\cdots\rightarrow (A\oplus (P_2\oplus\cdots\oplus
P_{N_0-1}))[N_0+2]\rightarrow A[N_0+1]\rightarrow
A[N_0-1]\rightarrow \cdots\rightarrow (A\oplus(P_2\oplus\cdots\oplus
P_{N_0-1})\oplus(P_3\oplus\cdots\oplus P_{N_0-2})\oplus\cdots\oplus
P_{\frac{N_0}{2}}\oplus
P_{{\frac{N_0}{2}+1}})[\frac{N_0}{2}]\rightarrow
(A\oplus(P_2\oplus\cdots\oplus
P_{N_0-1})\oplus(P_3\oplus\cdots\oplus P_{N_0-2})\oplus\cdots\oplus
P_{\frac{N_0}{2}}\oplus
P_{{\frac{N_0}{2}+1}})[\frac{N_0}{2}-1]\rightarrow \cdots\rightarrow
(A\oplus(P_2\oplus\cdots\oplus
P_{N_0-1})\oplus(P_3\oplus\cdots\oplus P_{N_0-2}))[2] \rightarrow
(A\oplus(P_2\oplus\cdots\oplus P_{N_0-1}))[1]\rightarrow
A\rightarrow \Bbbk^{\otimes N_0}\rightarrow 0$ for $N_0$ being even.

It is obvious that most terms of the above resolutions are made of
many brackets, in order to avoid some misunderstandings, we
stipulate the following: Given a concrete $N\in \mathbb{N}$, whether
the bracket appears or not is completely determined by the
subscripts of the first object and the last object in the bracket.
If the subscript of the first object is smaller than that of the
last object, then such bracket appears. Otherwise, the bracket does
not appear.

Now it is easy to see that the algebra constructed above is the
desired $\delta$-resolution determined algebra, where $\delta$ is
defined as follows:
$$\delta(i)=\left\{\begin{array}{llll}
\frac{i(N_0+1)}{N_0},  &  \mbox{ $i \equiv 0 (\textrm{mod} N_0)$,}\\
\frac{(i-1)(N_0+1)}{N_0}+1,  &  \mbox{ $i \equiv 1 (\textrm{mod} N_0)$,}\\
\cdots& \cdots \\
\frac{(i\!-\!N_0+1)(N_0+1)}{N_0}+N_0\!-\!1, \quad & \mbox{ $i\equiv
N_0\!-\!\!1 (\textrm{mod} N_0)$.}
\end{array}
\right..
$$

Therefore, we are done.
\end{proof}

Now we will point out that the algebra constructed in the proof of
(Lemma \ref{lem3} (iii)) gives an answer to the third question.

\begin{lemma}\label{lem2}
Let $A$ be the algebra constructed in the proof of (Lemma \ref{lem3}
(iii)) and $E(A)=\bigoplus_{i\geq0}\Ext_A^i(\Bbbk^{\otimes
N_0},\Bbbk^{\otimes N_0})$ the Yoneda algebra of $A$. Then $E(A)$ is
minimally generated by $\Ext_A^0(\Bbbk^{\otimes N_0},\Bbbk^{\otimes
N_0})$, $\Ext_A^1(\Bbbk^{\otimes N_0},\Bbbk^{\otimes N_0})$ and
$\Ext_A^{N_0}(\Bbbk^{\otimes N_0},\Bbbk^{\otimes N_0})$.
\end{lemma}
\begin{proof}
We first prove that $E(A)$ can be generated in degrees $0$, $1$ and
$N_0$. By hypothesis, the resolution map $\delta$ of $A$ is defined
as

$$\delta(i)=\left\{\begin{array}{llll}
\frac{i(N_0+1)}{N_0},  &  \mbox{ $i \equiv 0 (\textrm{mod} N_0)$,}\\
\frac{(i-1)(N_0+1)}{N_0}+1,  &  \mbox{ $i \equiv 1 (\textrm{mod} N_0)$,}\\
\cdots& \cdots \\
\frac{(i\!-\!N_0+1)(N_0+1)}{N_0}+N_0\!-\!1, \quad & \mbox{ $i\equiv
N_0\!-\!\!1 (\textrm{mod} N_0)$.}
\end{array}
\right..
$$
It is easy to see that $\delta(i)=\delta(i-N_0)+\delta(N_0)$ for all
$i\geq N_0$ and $\delta(i)=i$ for all $0\leq i\leq N_0-1$. By
(Proposition 3.6, \cite{GMMZ}), we have
$$\Ext_A^{i}(\Bbbk^{\otimes N_0},\Bbbk^{\otimes N_0})=(\Ext_A^{1}(\Bbbk^{\otimes
N_0},\Bbbk^{\otimes N_0}))^i$$ for $ 0\leq i\leq N_0-1$ and
$$\Ext_A^{i}(\Bbbk^{\otimes N_0},\Bbbk^{\otimes N_0})=(\Ext_A^{i}(\Bbbk^{\otimes N_0},\Bbbk^{\otimes N_0}))^k
\cdot\Ext_A^{j}(\Bbbk^{\otimes N_0},\Bbbk^{\otimes N_0})$$ for
$i=kN_0+j$, where $0\leq j\leq N_0-1$ and $k\in\mathbb{N}-\{0\}.$
Thus, $E(A)$ can be generated by $\Ext_A^0(\Bbbk^{\otimes
N_0},\Bbbk^{\otimes N_0})$, $\Ext_A^1(\Bbbk^{\otimes
N_0},\Bbbk^{\otimes N_0})$ and $\Ext_A^{N_0}(\Bbbk^{\otimes
N_0},\Bbbk^{\otimes N_0})$.

Now we claim that $\Ext_A^{N_0}(\Bbbk^{\otimes N_0},\Bbbk^{\otimes
N_0})$ can not be generated in lower degrees, i.e., $E(A)$ is
minimally generated by $\Ext_A^0(\Bbbk^{\otimes N_0},\Bbbk^{\otimes
N_0})$, $\Ext_A^1(\Bbbk^{\otimes N_0},\Bbbk^{\otimes N_0})$ and
$\Ext_A^{N_0}(\Bbbk^{\otimes N_0},\Bbbk^{\otimes N_0})$. In fact, it
suffices to prove that $\Ext_A^{N_0}(\Bbbk^{\otimes
N_0},\Bbbk^{\otimes N_0})$ can not be generated by
$\Ext_A^{1}(\Bbbk^{\otimes N_0},\Bbbk^{\otimes N_0})$. Note that
$$(\Ext_A^{1}(\Bbbk^{\otimes N_0},\Bbbk^{\otimes N_0}))^{N_0}=(\Ext_A^{1}(\Bbbk^{\otimes N_0},\Bbbk^{\otimes N_0})_{-1})^{N_0}\subseteq
\Ext_A^{N_0}(\Bbbk^{\otimes N_0},\Bbbk^{\otimes N_0})_{-N_0}.$$But
recall that $\delta(N_0)=N_0+1$, which implies that
$\Ext_A^{N_0}(\Bbbk^{\otimes N_0},\Bbbk^{\otimes
N_0})=\Ext_A^{N_0}(\Bbbk^{\otimes N_0},\Bbbk^{\otimes
N_0})_{-N_0-1}$. Thus, $(\Ext_A^{1}(\Bbbk^{\otimes
N_0},\Bbbk^{\otimes N_0}))^{N_0}=0$.

Therefore, we are done.
\end{proof}

\begin{corollary}\label{co1}
There does not exist a uniform bound of the generation degree for
the Yoneda algebras of $\delta$-Koszul algebras.
\end{corollary}
\begin{proof}
Suppose that we have a uniform bound of the generation degree for
the Yoneda algebras of $\delta$-Koszul algebras, say
$N\in\mathbb{N}$. Now let $N_0=N+1$ in the algebra constructed in
the proof of (Lemma \ref{lem3} (iii)). Then by Lemma \ref{lem2}, we
have that $E(A)$ is minimally generated by $\Ext_A^0(\Bbbk^{\otimes
N+1},\Bbbk^{\otimes N+1})$, $\Ext_A^1(\Bbbk^{\otimes
N+1},\Bbbk^{\otimes N+1})$ and $\Ext_A^{N+1}(\Bbbk^{\otimes
N+1},\Bbbk^{\otimes N+1})$, which is a contradiction.
\end{proof}

Now putting Lemmas \ref{lem1}, \ref{lem3}, \ref{lem2} and Corollary
\ref{co1} together, we have the following result, which is the main
result of this paper.

\begin{theorem}
We have the following statements.
\begin{enumerate}

\item Let $\delta: \mathbb{N}\rightarrow \mathbb{N}$ be a good set
map. Then
\begin{enumerate}
\item there exists a $\delta$-resolution determined algebra,
\item there exists a $\delta$-Koszul algebra.
\end{enumerate}
\item There does not exist any bound $N\in \mathbb{N}$, such that
the Yoneda algebras of all the $\delta$-Koszul algebras can be
generated in degrees in $0$, $1$, $\cdots$, and $N$.
\end{enumerate}
\end{theorem}

\medskip
\noindent{{\it{\bf{Acknowledgments\/}}}} The author would like to
give his thanks to the referee for his/her many valuable
suggestions, which improve the quality of the paper a lot.

\bigskip
\bibliographystyle{amsplain}

\end{document}